\documentclass[11pt,a4paper,twoside]{article}
\usepackage[lmargin=20mm,rmargin=20mm,vmargin=35mm]{geometry}
\usepackage{indentfirst}
\usepackage[cp1250]{inputenc}
\usepackage{amsfonts}
\usepackage{amssymb}
\usepackage{amsmath}
\usepackage{amsthm}
\usepackage{color}
\usepackage{secdot}

\newtheorem{tw}{Theorem}[section]
\newtheorem{q}[tw]{Question}

\newtheorem{prop}[tw]{Proposition}
\newtheorem{lem}[tw]{Lemma}
\newtheorem{cor}[tw]{Corollary}
\theoremstyle{definition}
\newtheorem{deff}[tw]{Definition}

\begin{document}

\title{APD profiles and transfinite asymptotic dimension}
\author{Kamil Orzechowski\\
Faculty of Mathematics and Natural Sciences\\
University of Rzesz\'ow\\
35-959 Rzesz\'ow, Poland\\
E-mail: kamil.orz@gmail.com}
%\date{}

\maketitle

\begin{abstract}
We develop the theory of APD profiles introduced by J. Dydak for $\infty$-pseudometric spaces (\cite{Dydak}). We connect them with transfinite asymptotic dimension defined by T. Radul (\cite{Rad}). We give a characterization of spaces with transfinite asymptotic dimension at most $\omega+n$ for $n\in\omega$ and a sufficient condition for a space to have transfinite asymptotic dimension at most $m\cdot \omega+n$ for $m,n\in\omega$, using the language of APD profiles.

\vspace{5mm}


\vspace{5mm}

\end{abstract}

\section{The set-theoretical background}

We begin with a set-theoretical definition due to P. Borst (\cite[Chapter II]{Borst}).
Let $L$ be an arbitrary set and $\mathrm{Fin} \, L$ the family of all finite, non-empty subsets of $L$.
For any $M\subset \mathrm{Fin}\, L$ and $\sigma\in \{\emptyset\}\cup \mathrm{Fin}\, L$ we put
$$M^{\sigma}=\{\tau \in \mathrm{Fin}\, L\colon \tau \cup \sigma\in M \, \text{and} \, \tau \cap \sigma=\emptyset \}.$$
We abbreviate $M^{\{a\}}$ by $M^a$ for any $a\in L$. 

\begin{deff}
Let $M$ be a subfamily of $\mathrm{Fin} \,L$. Define the ordinal number $\mathrm{Ord}\, M$ inductively as follows:\\
$\mathrm{Ord}\, M=0$ iff $M=\emptyset$;\\
$\mathrm{Ord}\,M\leq \alpha$ iff $\mathrm{Ord}\, M^{a}<\alpha$ for every $a\in L$;\\
$\mathrm{Ord}\,M=\alpha$ iff $\mathrm{Ord}\,M\leq \alpha$ and $\mathrm{Ord}\,M<\alpha$ is not true;\\
$\mathrm{Ord}\,M=\infty$ iff $\mathrm{Ord}\,M > \alpha$ for every ordinal number $\alpha$.
\end{deff}

In the case $\mathrm{Ord}\,M=\infty$ we can also say that the ordinal number $\mathrm{Ord} \,M$ {\em does not exist}.
We recollect some basic properties of the ordinal number $\mathrm{Ord}\, M$.

\begin{lem}[\textup{\cite[2.1.4]{Borst}}]\label{car}
Let $L$ be a set and $M\subset \mathrm{Fin} \,L$. In addition, let $n\in\omega$. Then $\mathrm{Ord}\, M\leq n$ if and only if $|\sigma|\leq n$ for every $\sigma\in M$.
\end{lem}

Thus one can say that $\mathrm{Ord} \,M$ is a transfinite generalization of the supremum of cardinalities of all members of $M$.

We call a subfamily $M$ of $\mathrm{Fin}\,L$ {\em inclusive} iff for every $\sigma, \sigma' \in \mathrm{Fin}L$ such that $\sigma' \subset \sigma$:  $\sigma \in M$ implies $\sigma'\in M$. By $\mathbb{N}$ we denote the set of all positive integers.

\begin{lem}[\textup{\cite[2.1.3]{Borst}}]\label{C}
Let $L$ be a set and $M$ an inclusive subfamily of $\mathrm{Fin}\,L$. Then $\mathrm{Ord}\, M=\infty$ iff there exists a sequence $\left(a_i\right)_{i\in \mathbb{N}}$ of distinct elements of $L$ such that $\sigma_n = \{a_i\colon 1\leq i \leq n\} \in M$ for each $n\in \mathbb{N}$.
\end{lem}

\begin{lem}[\textup{\cite[2.1.6]{Borst}}]\label{fun}
Let $\phi \colon L \to L'$ be a function and $M\subset \mathrm{Fin} \,L$, $M'\subset \mathrm{Fin}\, L'$ be such that for every $\sigma \in M$ we have $\phi(\sigma)\in M'$ and $|\phi(\sigma)|=|\sigma|$. Then $\mathrm{Ord} \,M \leq \mathrm{Ord}\, M'$.
\end{lem}

\begin{lem}[\textup{cf. \cite[Theorem 4]{Rad}}]
If $L$ is a countable set, $M\subset \mathrm{Fin} \,L$ and $\mathrm{Ord} \,M <\infty$, then $\mathrm{Ord} \,M<\omega_1$ (which means it is a countable ordinal number).
\end{lem}

We will prove another useful lemma (one direction comes from \cite[2.1.5]{Borst}).

\begin{lem}
Let $L$ be a set, $M\subset \mathrm{Fin}\, L$ and $\gamma \in\{\emptyset\}\cup\mathrm{Fin}\, L$. Let $\alpha>0$ be an ordinal number and $p\in\omega$.
Then $\mathrm{Ord}\, M^{\gamma}< \alpha + p$ if and only if $\mathrm{Ord}\, M^{\gamma\cup\sigma}< \alpha$ for every $\sigma\in \{\emptyset\}\cup\mathrm{Fin}\, L$ with $|\sigma|=p$ and $\gamma \cap \sigma =\emptyset$.
\end{lem}
\begin{proof}
We proceed by induction on $p$. If $p=0$, then $\sigma$ must be empty and the assertion is trivial.
Assume the lemma holds for $p-1$. Since $\alpha+p=(\alpha+1)+(p-1)$, we have $\mathrm{Ord}\, M^{\gamma}< \alpha + p$ if and only if $\mathrm{Ord}\, M^{\gamma\cup\sigma'}< \alpha+1$ for every $\sigma'$ with $|\sigma'|=p-1$, $\sigma'\cap \gamma =\emptyset$.

Suppose the latter and take an arbitrary $\sigma\in \{\emptyset\}\cup\mathrm{Fin}\, L$ with $|\sigma|=p$ and $\gamma \cap \sigma =\emptyset$. Pick an element $a\in\sigma$ and let $\sigma':=\sigma\setminus\{a\}$. We have $\mathrm{Ord} \, M^{\gamma\cup\sigma}=\mathrm{Ord}\, (M^{\gamma \cup \sigma'})^a$, which is by definition less than $\alpha$.

Suppose now that $\mathrm{Ord}\, M^{\gamma\cup\sigma}< \alpha$ for every $\sigma\in \{\emptyset\}\cup\mathrm{Fin}\, L$ with $|\sigma|=p$ and $\gamma \cap \sigma =\emptyset$. Take an arbitrary $\sigma'$ with $|\sigma'|=p-1$, $\sigma'\cap \gamma =\emptyset$. We want to show that $\mathrm{Ord}\, M^{\gamma\cup\sigma'}\leq \alpha$, i.e. $\mathrm{Ord}\, (M^{\gamma\cup\sigma'})^{a}<\alpha$ for any $a\in L$. It suffices to check it only for $a\not\in\gamma\cup\sigma'$. Then $(M^{\gamma\cup\sigma'})^{a}=M^{\gamma\cup \sigma}$ for $\sigma:=\sigma'\cup\{a\}$ and we can use our assumption.
\end{proof}

The most interesting case is that for $\gamma=\emptyset$.
\begin{cor}\label{mycor}
Let $L$ be a set, $M\subset \mathrm{Fin}\, L$, $\alpha>0$ and $p\in\omega$.
Then $\mathrm{Ord}\, M< \alpha + p$ if and only if $\mathrm{Ord}\, M^{\sigma}< \alpha$ for every $\sigma\subset L$ with $|\sigma|=p$.
\end{cor}

We are going to establish a condition classifying inclusive families $M\subset \mathrm{Fin}\, L$ with $\mathrm{Ord}\, M<\omega\cdot\omega$. It will be useful to introduce some (simplified) game-theoretical terminology.

\begin{deff}\label{st}
Let $m,n\in\omega$, $n>0$. We call a subset $S$ of the cartesian power $(\mathrm{Fin}\, L)^{m+1}$ an $m$-{\em strategy starting at $n$} if and only if it satisfies the following conditions:
\begin{enumerate}
\item $|\sigma_0|=n$ for any $(\sigma_0,\dots,\sigma_m)\in S$;
\item if $(\sigma_0,\dots,\sigma_m)\in S$, $(\tau_0,\dots,\tau_m)\in S$ and $\sigma_i=\tau_i$ for $0\leq i\leq k<m$, then $|\sigma_{k+1}|=|\tau_{k+1}|$;
\item if $0\leq k\leq m$ and $(\sigma_0,\dots,\sigma_k)$ is an initial segment of some $(\sigma_0,\dots,\sigma_m)\in S$, then so is $(\sigma_0,\dots,\sigma_{k-1},\tau)$ for any $\tau$ with $|\tau|=|\sigma_k|$.
\end{enumerate} 
\end{deff}
Loosely speaking, given $(\sigma_0,\dots,\sigma_k)$ already constructed following $S$, the strategy determines the cardinality of the next term $\sigma_{k+1}$.

\begin{prop}\label{strat}
Let $L$ be a set, $M$ an inclusive subfamily of $\mathrm{Fin}\, L$ and $m,n\in\omega$. Then {$\mathrm{Ord}\, M\leq m\cdot\omega+n$} if and only if there exists an $m$-strategy $S$ starting at $n+1$ such that for every sequence $(\sigma_0,\sigma_1,\dots,\sigma_m) \in S$ of pairwise disjoint sets the condition $\sigma_0\cup\dots\cup\sigma_m \not\in M$ holds.
\end{prop}
\begin{proof}
We proceed by induction on $m\in\omega$. For $m=0$ and fixed $n\in\omega$, we have $\mathrm{Ord}\,M\leq n$ iff the cardinality of all $\sigma\in M$ is bounded by $n$. Since $M$ is inclusive, it is equivalent to say that $\sigma_0\not\in M$ for any $\sigma_0$ with $|\sigma_0|=n+1$, i.e. the (unique) $0$-strategy starting at $n+1$ fullfills the requirements of the proposition.  

Let $m>0$ and suppose the proposition holds for $m-1$. Fix $n\in\omega$.
Assume $\mathrm{Ord}\, M \leq m\cdot\omega +n$. By Corollary \ref{mycor} we have $\mathrm{Ord}\, M^{\sigma_0} < m \cdot \omega$ for any $\sigma_0$ with $|\sigma_0|=n+1$. Thus, $\mathrm{Ord}\, M^{\sigma_0} \leq (m-1) \cdot \omega + n_1$ for some $n_1=n_1(\sigma_0)\in\omega$. By the inductive assumption, there is an $(m-1)$-strategy $S(\sigma_0)$ starting at $n_1+1$ such that $\sigma_1\cup \dots \sigma_m \not \in M^{\sigma_0}$ for every sequence $(\sigma_1,\dots,\sigma_m) \in S(\sigma_0)$ of pairwise disjoint sets.
Define $S$ to consist of all $(\sigma_0,\sigma_1,\dots,\sigma_m)$ such that $|\sigma_0|=n+1$ and $(\sigma_1,\dots,\sigma_m)\in S(\sigma_0)$. It is not hard to check that $S$ is an $m$-strategy starting at $n+1$ and satisfies $\sigma_0\cup\dots\cup\sigma_m \not\in M$ for pairwise disjoint $(\sigma_0,\sigma_1,\dots,\sigma_m) \in S$.

Suppose we have an $m$-strategy starting at $n+1$ and satisfying $\sigma_0\cup\dots\cup\sigma_m \not\in M$ for pairwise disjoint $(\sigma_0,\sigma_1,\dots,\sigma_m) \in S$. Then any $\sigma_0$ with $|\sigma_0|=n+1$ determines a number $n_1(\sigma_0)$ and an {$(m-1)$-strategy} $S(\sigma_0)$ starting at $n_1(\sigma_0)+1$ consisting of those $(\sigma_1,\dots,\sigma_m)$ for which $(\sigma_0,\sigma_1,\dots,\sigma_m)\in S$. Therefore $\sigma_1\cup \dots \sigma_m \not \in M^{\sigma_0}$ for every sequence $(\sigma_1,\dots,\sigma_m) \in S(\sigma_0)$ of pairwise disjoint sets. Hence, by the inductive assumption, $\mathrm{Ord}\, M^{\sigma_0} \leq (m-1) \cdot \omega + n_1(\sigma_0)<m\cdot \omega$. Since it holds for any $|\sigma_0|=n+1$, Corollary \ref{mycor} implies  $\mathrm{Ord}\, M \leq m\cdot\omega + n$.
\end{proof}

\section{Two approaches to transfinite asymptotic dimension}

The following definitions are usually formulated for metric spaces. However, we adjust them to the broader context of $\infty$-{\em pseudometric spaces}, i.e. spaces consisting of a set $X$ with a function $d\colon X\times X \to [0,\infty]$ satisfying the properties of symmetry, triangle inequality and taking value $0$ on the diagonal in $X\times X$.   

\begin{deff}\label{ub}
We say that a family $\mathcal{U}$ of subspaces of an $\infty$-pseudometric space $(X,d)$ is {\em uniformly bounded} if the number
$$\mathrm{mesh}(\mathcal{U}):=\sup \{\mathrm{diam}(U) \colon U \in \mathcal{U}\}$$
is finite. Let $r>0$, we say that $\mathcal{U}$ is $r$-{\em disjoint} if for any different $A,B\in \mathcal{U}$ we have
$$\mathrm{dist}(A,B):=\inf \{d(a,b)\colon a\in A, b\in B\} \geq r.$$ 
\end{deff}

\begin{deff}\label{scale}
Let $(X,d)$ be an $\infty$-pseudometric space, $x\in X$ and $r>0$. The {\em scale-}$r${\em -component} of $x$ in $X$ is the set of all points $y\in X$ that can be connected to $x$ by a scale-$r$-chain, i.e. a sequence of points $y=x_0, \dots, x_n=x$ in $X$ such that $B(x_i,r)\cap B(x_{i+1},r)\neq \emptyset$ for each $0\leq i <n$.

We say that $X$ is of {\em scale-}$r${\em -dimension} $0$ if the family of scale-$r$-components taken for all points $x\in X$ is uniformly bounded.
More generally, $X$ is of scale-$r$-dimension at most $n$ if it can be represented as the union of some $n+1$ subspaces of scale-$r$-dimension $0$.
\end{deff}

We associate with an $\infty$-pseudometric space $(X,d)$ two inclusive families of finite subsets of $\mathbb{N}$. The first of them is taken from \cite{Rad}.

\begin{deff}
We define $\mathcal{A}=\mathcal{A}(X,d)$ to consist precisely of all $\sigma\in \mathrm{Fin}\,\mathbb{N}$ such that there is no family $(X_i)_{i\in\sigma}$ of subspaces of $X$ which covers $X$ and each $X_i$ decomposes as the union of some $i$-disjoint uniformly bounded family.
\end{deff}

Similarly, we can define another family.
\begin{deff}
We define $M=M(X,d)$ to consist precisely of all $\sigma\in \mathrm{Fin}\,\mathbb{N}$ such that there is no family $(X_i)_{i\in\sigma}$ of subspaces of $X$ which covers $X$ and each $X_i$ has scale-$i$-dimension $0$.
\end{deff}

We have the following
\begin{prop}
For any $\infty$-pseudometric space $(X,d)$: $\mathrm{Ord} \, \mathcal{A}(X,d) = \mathrm{Ord} \, M(X,d)$.
\end{prop}
\begin{proof}
Observe that if $Y$ is the union of a $2r$-disjoint uniformly bounded family $\mathcal{U}$, then each scale-$r$-chain in $Y$ must lie in some common $U\in \mathcal{U}$, thus every scale-$r$-component of $Y$ has diameter bounded by $\mathrm{mesh}(\mathcal{U})$. So, if $\sigma\in M$, then $\{2n\colon n\in \sigma\}\in \mathcal{A}$. Applying Lemma \ref{fun} to the function $\phi\colon \mathbb{N}\to \mathbb{N}$, $\phi(n)=2n$, we conclude that $\mathrm{Ord}\, M \leq \mathrm{Ord}\, \mathcal{A}$.
For the converse, notice that different scale-$r$-components are $r$-disjoint so $\sigma\in \mathcal{A}$ implies $\sigma\in M$ and applying Lemma \ref{fun} to the identity function on $\mathbb{N}$ we obtain $\mathrm{Ord}\, \mathcal{A} \leq \mathrm{Ord} \, M$.
\end{proof}

\begin{deff}[\cite{Rad}]
We call the ordinal number $\mathrm{Ord} \, \mathcal{A}(X,d) = \mathrm{Ord} \, M(X,d)$ {\em transfinite asymptotic dimension} of $X$ and denote it by $\mathrm{trasdim}\, (X,d)$.
\end{deff}

\begin{deff}
We say that $X$ has {\em asymptotic property C} if and only if $\mathrm{trasdim}\, X <\infty$.
\end{deff}

Using Lemma \ref{C} and treating $\mathrm{trasdim}\, X <\infty$ as $\mathrm{Ord} \, \mathcal{A}$, we get original Dranishnikov's definition (\cite{Dran}), namely: $X$ has asymptotic property C if and only if for any sequence $\left(a_i\right)_{i\in \mathbb{N}}$ of distinct natural numbers there exists $n$ and a sequence $\left(\mathcal{U}_i\right)_{i=1}^{n}$ of uniformly bounded families such that $\bigcup_{i=1}^{n} \mathcal{U}_i$ covers $X$ and $\mathcal{U}_i$ is $a_i$-disjoint for $i=1,\dots,n$.

Thinking of $\mathrm{trasdim}\, X <\infty$ rather as of $\mathrm{Ord} \, M$, we get definition due to J. Dydak (\cite[5.12]{Dydak}):
$X$ has asymptotic property C if and only if for any sequence $\left(a_i\right)_{i\in \mathbb{N}}$ of distinct natural numbers there exists $n$ and a decomposition 
$X=\bigcup_{i=1}^{n} X_i$ such that each $X_i$ is of scale-$a_i$-dimension $0$.

\section{APD profiles and transfinite asymptotic dimension}

In his paper \cite{Dydak} J. Dydak defined so called {\em APD profile} of an $\infty$-pseudometric space. It is justified and convenient to deal only with integral APD profiles.
\begin{deff}
Suppose $X$ is an $\infty$-pseudometric space. A finite array of non-decreasing functions $(\alpha_0,\dots,\alpha_k)$ from $\mathbb{N}$ to $\mathbb{N}$ is an {\em integral APD profile} of $X$ if and only if $\alpha_0$ is constant and for any non-decreasing array $(r_0,\dots,r_k)$ of positive integers there is a decomposition
of $X$ as the union of its subsets $X_0,\dots,X_k$ such that each $X_i$ has scale-$r_i$-dimension at most $\alpha_i (r_{i-1})-1$ for $1\leq i\leq k$ and $0$ for $i=0$.
\end{deff}

It was proved in \cite{Dydak} that having an APD profile is a hereditary coarse invariant and so is the minimal length of APD profiles. A space $X$ has asymptotic dimension at most $n$ iff $(1,n)$ is an APD profile of $X$. Moreover, $X$ has finite asymptotic dimension iff it has an APD profile consisting of constant functions.
Spaces which admit an APD profile are said to have {\em asymptotic property D}, which implies having asymptotic property C.
We are interested in finding a deeper relation between the form of an APD profile and the precise value of $\mathrm{trasdim}\,X$.

\begin{tw}\label{moje}
Let $n\in\omega$. An $\infty$-pseudometric space $X$ has an integral APD profile $(n+1,f)$ if and only if $\mathrm{trasdim}\, X \leq \omega+n$.
\end{tw}
 
\begin{proof}
Suppose that $(n+1,f)$ is an integral APD of $X$. By definition, for any $r_0\leq r_1$ there exists a decomposition $X=Y_0 \cup Y_1$ such that $Y_0$ further decomposes as $Y_0=X_{1}\cup\dots \cup X_{n+1}$ and $Y_1$ as
$Y_1=X_{n+2} \cup \dots \cup X_{n+f(r_0)+1}$, where each $X_i$ is of scale-$r_0$-dimension $0$ for $1\leq i \leq n+1$ and of scale-$r_1$-dimension $0$ for $n+2\leq i\leq n+f(r_0)+1$.
Take a subset $\sigma\subset\mathbb{N}$ of cardinality $n+1$. List its elements in increasing order: $\sigma=\{a_1,\dots,a_{n+1}\}$. We will show that $\mathrm{Ord}\, M^{\sigma}< f(a_{n+1})<\omega$. Let $\tau=\{b_1,\dots,b_{f(a_{n+1})}\}$ (elements listed in increasing order) be a subset of $\mathbb{N}$ disjoint with $\sigma$. We claim that $\sigma \cup \tau\not\in M$. Let us take $r_0:=a_{n+1}$ and $r_1:=\max\left(a_{n+1},b_{f(a_{n+1})}\right)$. Then the family $(X_i)_{i=1}^{n+f(a_{n+1})+1}$ witnesses that $\sigma \cup \tau\not\in M$. Hence all members of $M^{\sigma}$ have cardinality less than $f(a_{n+1})$. Using Lemma \ref{car} and Corollary \ref{mycor}, we finish one part of the proof.

Suppose $\mathrm{trasdim}\, X \leq \omega+n$. For $k\in\mathbb{N}$ put $f(k):=\mathrm{Ord} \, M^{\{k,\dots,k+n\}} +1$. From Corollary \ref{mycor} we conclude that $f$ takes values in $\mathbb{N}$. It is not hard to check that $f$ is non-decreasing. We claim that $(n+1,f)$ is an integral APD profile of $X$. Fix natural numbers $r_0\leq r_1$. Consider the set $\tau=\{r_0,\dots,r_0 +n,m,m+1,\dots,m+f(r_0)-1\}$, where $m:=\max\left(r_1,r_0+n+1\right)$, and let $\tau'=\{m,m+1,\dots,m+f(r_0)-1\}$. The cardinality of $\tau'$ exceeds $\mathrm{Ord}\, M^{\{r_0,\dots,r_0+n\}}$ so $\tau'\not\in M^{\{r_0,\dots,r_0+n\}}$ and $\tau\not\in M$. The latter means that there exists a decomposition $X=X_{r_0}\cup\dots\cup X_{r_0+n} \cup X_m\cup \dots \cup X_{m + f(r_0)-1}$ such that each $X_i$ has scale-$i$-dimension $0$ for $i\in \tau$. 
In particular, all subspaces $X_{r_0}, \dots, X_{r_0+n}$ have scale-$r_0$-dimension $0$, hence they form one subspace of scale-$r_0$-dimension at most $n$. Similarly, 
the subspaces $X_{m}, \dots, X_{m+f(r_0)-1}$ form one subspace of scale-$r_1$-dimension at most $f(r_0)-1$. Thus $(n+1,f)$ is an APD profile of $X$. 
\end{proof}

\begin{cor}
 An $\infty$-pseudometric space satisfies $\mathrm{trasdim}\, X <\omega+\omega$ if and only if it has an integral APD profile consisting of two functions $(\alpha_0,\alpha_1)$.
\end{cor}

M. Satkiewicz formulated in \cite{Sat} the {\em omega conjecture} asserting that if $\omega \leq \mathrm{trasdim}\, X <\infty$, then $\mathrm{trasdim}\, X=\omega$. Recently, Y. Wu and J. Zhu (\cite{Wu}) have disproved it constructing a metric space with transfinite asymptotic dimension $\omega +1$. Our theorem implies that this space has an APD profile of the form $(2,f)$ for some $f$ but does not have an APD profile of the form $(1,g)$ for any $g$.

One direction of Theorem \ref{moje} can be generalized as follows:
\begin{tw}\label{great}
Let $n,m\in\omega$, $m>0$. If an $\infty$-pseudometric space $X$ has an integral APD profile $(n+1,\alpha_1,\dots,\alpha_m)$ for some functions $\alpha_1,\dots,\alpha_m$, then $\mathrm{trasdim}\, X \leq m\cdot \omega+n$.
\end{tw}
\begin{proof}
Define $S\subset (\mathrm{Fin}\,\mathbb{N})^{m+1}$ to consist of all $(\sigma_0,\dots,\sigma_m)\in(\mathrm{Fin}\,\mathbb{N})^{m+1}$ such that $|\sigma_0|=n+1$ and
$|\sigma_k|=\alpha_k \left(\max \bigcup_{i=0}^{k-1} \sigma_i\right)$ for $k=1,\dots,m$. It is easy to check that $S$ is an $m$-strategy starting at $n+1$.
According to Proposition \ref{strat}, it is sufficient to show that for every sequence $(\sigma_0,\sigma_1,\dots,\sigma_m) \in S$ of pairwise disjoint sets the condition $\sigma_0\cup\dots\cup\sigma_m \not\in M(X,d)$ holds.
Fix $(\sigma_0,\sigma_1,\dots,\sigma_m) \in S$ with pairwise disjoint terms. Put $r_k:=\max \bigcup_{i=0}^{k}\sigma_i$ for $k=0,\dots,m$. Obviously $r_0\leq\dots\leq r_m$. Applying the definition of an APD profile, we obtain a decomposition $X=X_0 \cup \dots \cup X_m$ such that $X_k$ is of scale-$r_k$-dimension at most $\alpha_k (r_{k-1})-1=\alpha_k \left(\max \bigcup_{i=0}^{k-1}\sigma_i\right)-1=|\sigma_k|-1$, for $1\leq k\leq m$, and at most $n$ for $k=0$. Thus, each $X_k$ decomposes further as the union of $|\sigma_k|$ subspaces, each of them being of scale-$r_k$-dimension $0$. We can write this decomposition as $X_k=\bigcup_{j\in\sigma_k} X_{k,j}$.
Since $r_k=\max \bigcup_{i=0}^{k}\sigma_i$, each $X_{k,j}$ is in particular of scale-$j$-dimension $0$. Combining all the $X_k$, we get a decomposition 
$X=\bigcup_{i\in\sigma_0\cup \dots \cup \sigma_m} X_i$ where every $X_i$ is of scale-$i$-dimension $0$. That means that $\sigma_0 \cup \dots \cup \sigma_m \not\in M(X,d)$, as desired.
\end{proof}

We have the following
\begin{cor}
If an $\infty$-pseudometric space $X$ has asymptotic property D, then $\mathrm{trasdim}\, X < \omega\cdot\omega$.
\end{cor}

The strategy $S$ constructed while proving Theorem \ref{great} is rather special because the cardinality of $\sigma_k$ is determined somewhat ``uniformly'' by given $\alpha_k$ (taking an expression of previous $\sigma_0,\dots,\sigma_{k-1}$ as its argument). In general, a strategy in Definition \ref{st} does not provide such a uniform rule, it is merely a whole {\em strategy tree} of the game in which we react with a natural number to a given sequence of subsets (and such a reaction could depend very ``wildly'' on a current position).

\begin{q}
Is there a space $X$ with $\mathrm{trasdim}\, X <\omega\cdot\omega$ but without asymptotic property D?
\end{q}

If the answer were affirmative, there would be a space with asymptotic property C but without aymptotic property D (thus responding the question in \cite[5.15]{Dydak}).

\end{document}